\documentclass[11pt, reqno]{amsart}

\usepackage{graphicx}
\usepackage{amsmath,amssymb,mathtools,amsopn,amsfonts,amsthm}
\usepackage[shortlabels]{enumitem}
\usepackage{float}
\usepackage{bbm}
\usepackage{subfigure}
\usepackage{eepic}
\usepackage{tikz}
\usetikzlibrary{positioning}
\usepackage{a4wide}
\usepackage[utf8]{inputenc}
\usepackage{epsfig}
\usepackage{latexsym}
\usepackage[UKenglish]{babel}
\usepackage{verbatim}
\usepackage[initials]{amsrefs}
\usepackage[bookmarks]{hyperref}
\usetikzlibrary{decorations.markings}

\newtheorem{thm}{Theorem}[section]
\newtheorem{lma}[thm]{Lemma}

\newtheorem{cor}[thm]{Corollary}
\newtheorem{prop}[thm]{Proposition}
\theoremstyle{definition}
\newtheorem{defn}[thm]{Definition}

\newtheorem{rem}[thm]{Remark}

\newcommand{\N}{\mathbb{N}}

\newcommand{\I}{\mathcal{I}}

\renewcommand{\i}{\mathtt{i}}
\renewcommand{\j}{\mathtt{j}}
\renewcommand{\k}{\mathtt{k}}
\renewcommand{\l}{\mathtt{l}  }
\newcommand{\lang}{\Sigma^*}
\newcommand{\G}{\mathcal{G}}
\renewcommand{\L}{\mathcal{L}}
\newcommand{\C}{\mathcal{C}}
\renewcommand{\O}{\mathcal{O}}
\newcommand{\m}{\mathtt{m}}

\newcommand{\hd}{\dim_\textup{H}}
\newcommand{\bd}{\dim_\textup{B}}
\newcommand{\ubd}{\overline{\dim}_\textup{B}}
\newcommand{\lbd}{\underline{\dim}_\textup{B}}

\renewcommand{\m}{\mathfrak{m}}
\newcommand{\n}{\mathfrak{n}}

\title{Box dimensions of $(\times \m, \times \n)$-invariant sets}

\author{Jonathan M. Fraser} \address{Mathematical Institute, University of St Andrews, Scotland, KY16 9SS}
\email{jmf32@st-andrews.ac.uk}
\author{Natalia Jurga} \address{Mathematical Institute, University of St Andrews, Scotland, KY16 9SS}
\email{naj1@st-andrews.ac.uk}

\thanks{The  authors were both     supported by an \emph{EPSRC Standard Grant} (EP/R015104/1). J. M. Fraser was also supported by  a  \emph{Leverhulme Trust Research Project Grant} (RPG-2019-034).}

\begin{document}

\maketitle

\begin{abstract}
We study the box dimensions of sets invariant under the toral endomorphism $(x, y) \mapsto (\m x \text{ mod } 1, \, \n y \text{ mod } 1)$ for integers $\mathfrak{n}>\mathfrak{m} \geq 2$.  The basic examples of such sets are Bedford-McMullen carpets and, more generally, invariant sets are modelled by subshifts on the associated symbolic space.  When this subshift is topologically mixing  and sofic the situation is well-understood by results of Kenyon and Peres.    Moreover, other work of Kenyon and Peres shows that the Hausdorff dimension is generally given by a variational principle.  Therefore, our work is focused on the box dimensions in the case where the underlying shift is not topologically mixing and  sofic.  We establish straightforward upper and lower bounds for the box dimensions in terms of entropy which hold for all subshifts and show that the upper bound is the correct value for coded subshifts whose entropy can be realised by words which can be freely concatenated, which includes many well-known families such as $\beta$-shifts,  (generalised) $S$-gap shifts, and transitive sofic shifts. We also provide examples of transitive coded subshifts where the general upper bound fails and the box dimension is actually given by the general lower bound.  In the non-transitive sofic setting, we provide a formula for the box dimensions which is often intermediate between the general lower and upper bounds.  
\end{abstract}

\section{Introduction}

We study compact sets invariant under the toral endomorphism
\[
T(x, y) = (\m x \text{ mod } 1, \, \n y \text{ mod } 1)
\]
 for integers $\n>\m \geq 2$.  This dynamical system is a basic and fundamental example of an expanding  non-conformal system and invariant sets have many subtle properties.  The simplest  examples of such invariant sets are the self-affine carpets  introduced by Bedford and McMullen in 1984 \cite{bedford, mcmullen}.  In particular, these are modelled by a full shift.  More generally, compact $(\times \m, \times \n)$-invariant sets are modelled by subshifts on the associated symbolic space.   Kenyon and Peres \cite{kp}  studied the more general case when  this subshift is topologically mixing and sofic and in \cite{kp-measures} they resolved the Hausdorff dimension case in general by proving a variational principle.  These papers provide the starting point for our investigation,   which is focused on the box dimensions in the case where the underlying shift is not topologically mixing and  sofic.  We expand  the theory in several directions.

Let $\Delta_{\m,\n}=\{(a,b)\; : \; 1 \leq a \leq \m, \; 1 \leq b \leq \n, \; a,b \in \mathbb{N}\}.$  For any $(a,b) \in \Delta_{(\m,\n)}$ define the contraction $S_{(a,b)}:[0,1]^2 \to [0,1]^2$ as
$$S_{(a,b)}(x,y)= \begin{pmatrix} \frac{1}{\m}&0 \\0&\frac{1}{\n} \end{pmatrix} \begin{pmatrix} x\\y \end{pmatrix}+ \begin{pmatrix} \frac{a-1}{\m}\\\frac{b-1}{\n} \end{pmatrix}.$$
Define the coding map $\Pi:\Delta_{\m,\n}^\N \to [0,1]^2$ as
$$\Pi\left((a_1,b_1)(a_2,b_2) \ldots\right):=\lim_{n \to \infty} S_{(a_1,b_1)} \circ \cdots \circ S_{(a_n,b_n)}(0).$$
Consider any compact $(\times \m, \times \n)$-invariant set $F$, meaning that $T(F) \subseteq F$. Then there exists a digit set $\I \subseteq\Delta_{\m,\n}$ and a subshift $\Sigma$ on the digit set $\I$ (meaning a compact $\sigma$-invariant subset $\Sigma \subseteq \I^\N$, i.e. $\sigma(\Sigma) \subseteq \Sigma$ where $\sigma:\Sigma \to \Sigma$ denotes the left shift map) such that $F=\Pi(\Sigma)$. For example, if $\Sigma$ is the full shift on $\I$ then $\Pi(\Sigma)$ is a Bedford-McMullen carpet \cite{bedford,mcmullen}. For brevity, rather than writing sequences in $\Sigma$ as $(a_1,b_1)(a_2,b_2) \ldots$ and finite words which appear in sequences of $\Sigma$ as $(a_1,b_2) \ldots (a_n,b_n)$ we will for the most part denote both infinite sequences and finite words by variables such as $\i,\j,$ and $\k$.

Given a subshift $\Sigma$, let $\lang$ denote the language of $\Sigma$, meaning the collection of finite words which appear in sequences $\i \in \Sigma$. For $n \in \N$ let $\Sigma_n$ denote words in $\lang$ which have length $n$. We say $\Sigma$ is \emph{topologically transitive} if for all $\i, \j \in \lang$ there exists $\k \in \lang$ such that $\i\k\j \in \lang$. We say $\Sigma$ is \emph{topologically mixing} if there exists $N \in \N$ such that for all $\i, \j \in \lang$ there exists $\k \in \Sigma_N$ such that $\i \k \j \in \lang$. Recall that the \emph{topological entropy} of $\Sigma$ is defined as $h(\Sigma):=\lim_{n \to \infty} \frac{1}{n} \log \# \Sigma_n$, where the limit exists by submultiplicativity arguments.

The $(\times \m, \times \n)$-invariant sets are typically fractal and a key question of interest is in computing their dimensions, especially Hausdorff and box dimensions, see \cite{bedford,deliu,  kp, kp-measures,mcmullen}.  For more background on Hausdorff and box dimensions, see \cite{falconer}.   We write $\hd$, $\lbd$, and $\ubd$ for the Hausdorff, lower and upper box dimensions, respectively.  The lower and upper box dimensions are defined by
\[
 \lbd E  = \liminf_{\delta \to 0} \frac{\log N_\delta(E)}{-\log \delta} \qquad \text {and} \qquad \ubd E = \limsup_{\delta \to 0} \frac{\log N_\delta(E)}{-\log \delta},
\]
respectively, where $N_\delta(E)$ denotes the smallest number of sets of diameter $\delta>0$ required to cover $E$.  It is useful to keep in mind that, for all bounded sets $E$ in Euclidean space, 
\[
\hd E \leq \lbd E \leq \ubd E.
\]
Moreover, if the upper and lower box dimensions coincide we simply refer to the box dimension, written $\bd$.  In the case where $\Sigma$ is a full shift (over a restricted alphabet $\mathcal{I}\subseteq\Delta_{\m,\n}$), the box and Hausdorff dimensions were computed independently by Bedford \cite{bedford} and McMullen \cite{mcmullen}. If $\Sigma$ is a topologically mixing sofic subshift, then the box and Hausdorff dimensions were given by Kenyon and Peres \cite{kp}.  We say that a subshift is \emph{sofic} if it can be presented by a finite directed labelled graph $G$ (see Section \ref{sofic} for a more precise definition).  If $\Sigma$ is a topologically transitive subshift of finite type, then the box dimension was computed by Deliu \emph{et al} \cite{deliu}.  The only progress beyond the sofic setting is provided by Kenyon and Peres \cite{kp-measures} where they show that for any compact $(\times \m, \times \n)$-invariant set  the Hausdorff dimension is given by a variational principle, that is, as the supremum of the Hausdorff dimensions of  $(\times \m, \times \n)$-invariant measures supported on the set.  It is also shown that there exists a maximising (ergodic) measure, which achieves the Hausdorff dimension of the set.  Moreover, it is shown in \cite{kp-measures} that  the Hausdorff dimension  of an ergodic $(\times \m, \times \n)$-invariant measure is given by a Ledrappier-Young formula. In some sense, this settles the question of Hausdorff dimension.  The box dimensions of  $(\times \m, \times \n)$-invariant sets remains an interesting open programme.  We recall the box dimension result of Kenyon and Peres which is the current state of the art. Let $\pi: \Sigma \to \pi \Sigma$ denote the projection mapping $\pi\left((a_1,b_1)(a_2,b_2) \ldots\right)=a_1a_2 \ldots$.  In particular, $\pi \Sigma$ is itself a subshift.

\begin{thm}[Proposition 3.5, \cite{kp}]\label{kp}
Suppose $\Sigma$ is a topologically mixing sofic subshift. Then
\begin{eqnarray}
\bd \Pi(\Sigma)=\frac{h(\pi\Sigma)}{\log \m} + \frac{h(\Sigma) - h(\pi\Sigma)}{\log \n}.\label{usual}
\end{eqnarray}
\end{thm}

It is straightforward to construct an example where \eqref{usual} does not hold for a general sofic subshift $\Sigma$. For example fix $\m=2$, $\n=4$ and $\I=\{(1,1),(1,2),(1,3),(1,4),(2,1)\}$ and denote $\Sigma_2=  \{(1,1),(2,1)\}^\N$, $\Sigma_3=\{(1,2),(1,3),(1,4)\}^\N$. Consider the subshift of finite type $\Sigma=\Sigma_2 \cup \Sigma_3$. Then,
$$\bd \Pi(\Sigma)=\max\left\{\bd \Pi(\Sigma_2),\bd \Pi(\Sigma_3)\right\}=1<1+\frac{\log 3-\log 2}{\log 4}=\frac{h(\pi\Sigma)}{\log \m} + \frac{h(\Sigma) - h(\pi\Sigma)}{\log \n}$$
where in the second equality we apply \eqref{usual} to $\bd \Pi(\Sigma_2)$ and $\bd \Pi(\Sigma_3)$. This example heavily relies on a lack of transitivity. 

We fully resolve the  sofic case by finding a formula that holds for any sofic subshift (which is not just the maximum over irreducible parts as above) and which simplifies to \eqref{usual} in the transitive case, thus generalising Theorem \ref{kp} from topologically mixing to topologically transitive. 

We say a graph $G$ is \emph{irreducible} if given any pair of vertices $v,w \in G$ there is a path in $G$ from $v$ to $w$. Given a finite directed labelled graph $G$ which presents $\Sigma$, let $\{G_i\}_{i=1}^k$ denote the irreducible components of $G$, meaning the maximal irreducible subgraphs of $G$. Each subgraph $G_i$ therefore presents a subshift $\Sigma_{G_i} \subseteq \Sigma$. Given $1 \leq i \leq k$ we let $\{i\}^+$ denote the set of all indices $1 \leq j \leq k$ such that there is a path in $G$ from a vertex in $G_i$ to a vertex in $G_j$, noting that $\{i\}^+$ is necessarily non-empty since we always have $i \in \{i\}^+$.
%Similarly, we let $\{i\}^-$ denote the set of all indices $1 \leq j \leq k$ such that there is a path in $G$ from a vertex in $G_j$ to a vertex in $G_i$. 

\begin{thm} \label{red-thm}
Let $\Sigma$ be a sofic subshift which is presented by a graph $G$. Let $\{G_1, \ldots, G_k\}$ be the irreducible components of $G$. Then
\begin{equation}
\dim_\textup{B} \Pi(\Sigma)= \max_{1 \leq i \leq k}\left\{ \frac{h(\Sigma_{G_i})}{\log \n} + \max_{j \in \{i\}^+} h(\pi \Sigma_{G_j}) \left(\frac{1}{\log \m}-\frac{1}{\log \n}\right)\right\}. \label{red}
\end{equation}
\end{thm}

As in \cite[Proposition 3.5]{kp}, each entropy $h(\Sigma_{G_i})$ and $h(\pi \Sigma_{G_i})$ can be expressed in terms of the spectral radius of the adjacency matrix of an appropriate right-resolving presentation (of $\Sigma_{G_i}$ and $\pi \Sigma_{G_i}$ respectively). When $\Sigma$ is topologically transitive and sofic, $\Sigma$ can be presented by an irreducible labelled graph, therefore \eqref{red} simplifies to \eqref{usual}. Additionally, we can also recover \eqref{usual} for some sofic subshifts which are not topologically transitive, under some assumptions on the ``position'' of the entropy maximising irreducible components, see Corollary \ref{source}. Moreover, the ``position'' of the entropy maximising irreducible components can also determine whether or not the Hausdorff and box dimensions are equal, see Corollary \ref{b=h}.

Next, we turn to more general subshifts. By bounding $\lbd \Pi(\Sigma)$ (and $\hd \Pi(\Sigma)$) below by the box dimension of its projection and by a crude estimate involving entropy and the larger Lyapunov exponent, we show (see Proposition \ref{bounds}) that any invariant set satisfies a trivial lower bound of $\lbd \Pi(\Sigma) \geq \max\left\{\frac{h(\pi\Sigma)}{\log \m}, \frac{h(\Sigma)}{\log \n}\right\}$. On the other hand, we  also show  (see Proposition \ref{bounds}) that the right hand side of \eqref{usual} is a trivial upper bound on $\ubd \Pi(\Sigma)$ in general. While Theorem \ref{red-thm} demonstrates that the box dimension can drop from this trivial upper bound if $\Sigma$ is not topologically transitive, it is interesting to ask whether transitivity is sufficient for \eqref{usual} to hold for general subshifts. We answer this in the negative:

\begin{thm}\label{dimdrop}
There exists a topologically transitive subshift $\Sigma$ with $0<h(\pi\Sigma)< h( \Sigma)$ and
$$\bd \Pi(\Sigma)= \max\left\{\frac{h(\Sigma)}{\log \n}, \frac{h(\pi \Sigma)}{\log \m}\right\}.$$
\end{thm}

In particular, in the above example the trivial lower bound is in fact the exact value of the box dimension.  Moreover this box dimension is clearly strictly smaller than the trivial upper  bound  and we can modify our example such that either of the trivial lower bounds equals the box dimension.  The subshift $\Sigma$ that we construct towards the proof of Theorem \ref{dimdrop} falls into the class of \emph{coded subshifts}.  Coded subshifts, which were first introduced in \cite{bh} and include the well-known subclasses of $S$-gap shifts, $\beta$-shifts and Dyck shifts, are subshifts which can be presented by an irreducible (but not necessarily finite), directed labelled graph (see Section \ref{coded}). In particular, they clearly extend the class of transitive sofic subshifts and provide a natural and interesting class to investigate which, unlike subshifts of finite type and sofic subshifts in general, cannot be handled by techniques that depend on finiteness of the presentation.

A useful equivalent characterisation of coded subshifts is that a subshift $\Sigma$ is coded if there exists a countable collection of finite words $\C$, which we call generators, such that $\Sigma$ is the closure of the set of sequences obtained by freely concatenating the generators. In particular, $\pi \Sigma$ is also a coded subshift which is generated by $\pi \C$. We say that a coded subshift $\Sigma$ has \emph{unique decomposition with respect to $\C$} if no finite word can be written as a concatenation of generators in $\C$ in distinct ways. 

We will show that if the entropy of a coded subshift $\Sigma$ and $\pi \Sigma$ can be realised by counting words which can be obtained by concatenating their (respective) generators, then the box dimension  $\bd \Pi(\Sigma)$ equals the trivial upper bound given in Proposition \ref{bounds}. In particular let $\G_n$ denote all words of length $n$ in $\Sigma^*$ which can be written by concatenating generators from $\C$. Analogously, $\pi \G_n$ are all words of length $n$  in $(\pi \Sigma)^*$ which  can be written by concatenating generators from $\pi \C$. We denote
\begin{eqnarray*}
h:=\limsup_{n \to \infty} \frac{1}{n} \log \# \G_n \quad \quad & \textnormal{and} & \quad \quad  h_\pi:=\limsup_{n \to \infty} \frac{1}{n} \log \# \pi\G_n.
\end{eqnarray*}

\begin{thm}
Let $\Sigma$ be a coded subshift and suppose $h=h(\Sigma)$ and $h_\pi=h(\pi \Sigma)$. Then
\begin{eqnarray}
\bd \Pi(\Sigma)=\frac{h(\pi\Sigma)}{\log \m} + \frac{h(\Sigma) - h(\pi\Sigma)}{\log \n}.
\end{eqnarray}
\label{g-thm}
\end{thm}

Note that the example constructed in Theorem \ref{dimdrop} satisfies $h<h(\Sigma)$. A drawback of Theorem \ref{g-thm} is that in general it may not be straightforward to verify the equalities  $h=h(\Sigma)$ and $h_\pi=h(\pi \Sigma)$. However, under the assumption of unique decomposition of $\Sigma$ and $\pi \Sigma$ we provide a more practical way of checking that the conclusion of Theorem \ref{g-thm} holds. This is based on the fact that under the assumption of unique decomposition of $\Sigma$ and $\pi \Sigma$ (with respect to $\C$ and $\pi \C$), $h$ and $h_\pi$ can be understood as the Gurevic entropies of countable graphs associated with the coded subshifts $\Sigma$ and $\pi \Sigma$ (see Section \ref{coded}). This allows us to employ classical tools from the theory of countable Markov shifts which yields checkable criteria for Theorem \ref{g-thm} to hold, see Theorem \ref{fgeq1} below, whose statement requires the introduction of some further notation.

Let $\L_n$ denote all words of length $n$ in $\Sigma^*$ which appear at the beginning or end of some generator in $\C$, analogously $\pi \L_n$ are all words of length $n$ which appear at the beginning or end of some generator in $\pi \C$. We denote
\begin{eqnarray*}
\ell:=\limsup_{n \to \infty} \frac{1}{n} \log \# \L_n \quad \quad & \textnormal{and} & \quad \quad  \ell_\pi:=\limsup_{n \to \infty} \frac{1}{n} \log \# \pi\L_n.
\end{eqnarray*}

 Let $\C_n$ denote words in $\C$ of length $n \in \N$, analogously $\pi \C_n$ denotes words in $\pi \C$ of length $n$. Finally, define functions $f, f_\pi: [0, \infty) \to (0,\infty]$ by 
\begin{eqnarray} \label{f}
f(x)= \sum_{n=1}^\infty \# \C_ne^{-nx} \; \; \; \; & \textnormal{and} & \; \; \; \; f_{\pi}(x)= \sum_{n=1}^\infty \# \pi \C_n e^{-nx}.
\end{eqnarray}

\begin{thm} \label{fgeq1}
Suppose $\Sigma$ is a coded subshift such that $\Sigma$ and $\pi \Sigma$ have unique decomposition with respect to $\C$ and $\pi \C$ respectively. Additionally, assume $f(\ell)> 1$ and $f_{\pi}(\ell_\pi) >  1$. Then
$$ \bd \Pi(\Sigma)=\frac{h(\pi\Sigma)}{\log \m} + \frac{h(\Sigma) - h(\pi\Sigma)}{\log \n}.$$
\end{thm}

The usefulness of Theorem \ref{fgeq1} lies in the fact that $\# \C_n$, $\# \pi \C_n$, $\ell$ and $\ell_\pi$ are often easy to compute, which we demonstrate by applying it to generalised $S$-gap shifts in \S \ref{good}. We also note that Theorem \ref{fgeq1} can easily be adapted to allow $\pi \Sigma$ to be uniquely decomposing with respect to an arbitrary generating set $\C^\pi$ rather than $\pi \C$. In particular if  $\#\pi \C_n$ is replaced by $\# \C^\pi$ (words of length $n$ in $\C^\pi$) in the definition of $f_\pi$, then Theorem \ref{fgeq1} remains true under the assumption that $\pi \Sigma$ satisfies unique decomposition with respect to $\C^\pi$.
% The assumption  $f(h(L))\geq 1$ (and $f_{\pi}(h( L_{\pi})) \geq  1$) are natural in the setting of coded subshifts, see ???, and should be interpreted as saying that the entropy comes from concatenations of generators, not from the `boundary'.  Moreover, we provide examples  demonstrating  that these assumptions on the entropy of $L$ and $L_\pi$ are necessary, see Sections \ref{sdimdrop} and \ref{other}.

\section{Preliminaries}

We write $a \lesssim b$ to mean there exists a constant $C>0$ such that $a \leq Cb$.  The implicit constant $C$ may depend on parameters which are fixed in the hypotheses, such as $m, n$ and $\Sigma$, but crucially do not depend on variables in the proofs, such as the covering scale $\delta$. If we wish to emphasise that the $C$ depends on something else, not fixed in the hypothesis such as $\varepsilon$, then we write $a \lesssim_\varepsilon b$. Similarly, we write $a \gtrsim b$ to mean $b \lesssim a$ and $a \approx b$ to mean $a \lesssim b$ and $a \gtrsim b$ both hold (analogously $a \gtrsim_\varepsilon b$ and $a \approx_\varepsilon b$).  For $\i \in \Sigma_k$, we write $[\i]$ for the cylinder consisting of elements of $\Sigma$ with prefix $\i$.  We also refer to $\Pi([\i])$ as cylinders, although these are subsets of the fractal, rather than the symbolic space. Given $\i \in \Sigma$ or $\i \in \lang$ of length at least $n+1 \geq 2$ we let $\i|_n$ denote the truncation of $\i$ to its first $n$ digits. We also write $\# A$ to denote the cardinality of a (usually finite)  set $A$.
 
Let $\delta>0$. Throughout the paper we will let $k(\delta)$ denote the unique positive integer satisfying $\n^{-k(\delta)} \leq \delta < \n^{1-k(\delta)} $ and $l(\delta)$ denote the unique positive integer satisfying $\m^{-l(\delta)} \leq \delta < \m^{1-l(\delta)}$, noting that $k(\delta) < l(\delta)$ for sufficiently small $\delta$. Observe that by definition  $l(\delta) \approx \frac{-\log \delta}{\log \m}$ and $k(\delta) \approx \frac{-\log \delta}{\log \n}$ for sufficiently small $\delta$.

Here we prove the trivial lower and upper bounds that we alluded to in the introduction. The general strategy of relating covers to allowed words in $\Sigma$ and $\pi\Sigma$ will underpin all of our subsequent proofs, therefore we take care to include all of the details here.

\begin{prop} \label{bounds}
For all subshifts $\Sigma \subset \Delta_{\m,\n}^\N$,
\[
\max\left\{\frac{h(\pi\Sigma)}{\log \m}, \frac{h(\Sigma)}{\log \n}\right\} \leq  \hd \Pi(\Sigma) \leq \lbd \Pi(\Sigma)  \leq   \ubd \Pi(\Sigma) \leq \frac{h(\pi\Sigma)}{\log \m} + \frac{h(\Sigma) - h(\pi\Sigma)}{\log \n}.
\]
\end{prop}

\begin{proof}
Fix $\varepsilon>0$ and $\delta>0$.

 We begin with the upper bound. Consider $\Sigma_{k(\delta)}$ and  consider covers of the level $k(\delta)$ cylinders, $\Pi([\i])$,  independently.  For $\i \in \Sigma_k$, write
\[
M(\i, l) = \# \pi(\j \in \Sigma_l : \j|_{k(\delta)}=\i)
\]
for the number of children of $\i$ at level $l >k(\delta)$ which lie in distinct columns.  Then
\begin{eqnarray*}
N_\delta(\Pi(\Sigma)) &\approx & \sum_{\i \in \Sigma_{k(\delta)}} N_\delta(\Pi([\i])) \\ \\
&\approx  &   \sum_{\i \in \Sigma_{k(\delta)}}  M(\i, l(\delta)) \\ \\
 &\leq   &   \sum_{\i \in \Sigma_{k(\delta)}}  \#\pi \Sigma_{l(\delta)-k(\delta)}  \qquad \text{(using shift invariance)}\\ \\
&= & \# \Sigma_{k(\delta)}  \#\pi \Sigma_{l(\delta)-k(\delta)}  \\ \\
&\lesssim_{\varepsilon} & \exp((h(\Sigma)+\varepsilon)k(\delta))  \exp((h(\pi \Sigma)+\varepsilon)(l(\delta)-k(\delta))). 
\end{eqnarray*}
In particular since $l(\delta) \approx \frac{-\log \delta}{\log \m}$ and $k(\delta) \approx \frac{-\log \delta}{\log \n}$ we have
\begin{eqnarray*}
\frac{\log N_\delta(\Pi(\Sigma))}{-\log \delta} &\lesssim_\varepsilon& \frac{(h(\Sigma)+\varepsilon)\frac{-\log \delta}{\log \n}}{-\log \delta}+  \frac{(h(\pi \Sigma)+\varepsilon)(\frac{-\log \delta}{\log \m}- \frac{-\log \delta}{\log \n})}{-\log \delta},
\end{eqnarray*}
therefore letting $\delta \to 0$ yields the desired upper bound since $\varepsilon>0$ was chosen arbitrarily.

For the lower bounds, first observe that $\lbd \Pi(\Sigma) \geq  \hd \Pi(\Sigma) \geq   \hd \pi\Pi(\Sigma)= \frac{h(\pi\Sigma)}{\log \m}$, where the second inequality follows since the projection $\pi:[0,1]^2 \to [0,1]$ to the first coordinate is Lipschitz, and the final equality follows from  Furstenberg's result expressing the Hausdorff dimension of a subshift in terms of entropy \cite{furstenberg}. To see the second lower bound, let $\mu$ be a measure of maximal entropy for $\Sigma$ projected onto $\Pi(\Sigma)$. %exists using variational principle and that $\Sigma$ is compact}. 
  Let $\varepsilon>0$ be fixed and let $\delta>0$.  A ball of radius $\delta>0$ centred in $\Pi(\Sigma)$ intersects at most $\lesssim 1$ many level $k(\delta)$ cylinders each with mass
  $$\lesssim_\varepsilon \exp(-k(\delta)h(\Sigma)(1-\varepsilon)).$$
Therefore since  $k(\delta) \approx \frac{-\log \delta}{\log \n}$ we deduce that $\lbd \Pi(\Sigma) \geq \hd \Pi(\Sigma) \geq  \frac{h(\Sigma)}{\log \n}$ by the mass distribution principle, upon letting $\varepsilon \to 0$. 
\end{proof}

\section{Sofic $(\times \m, \times \n)$-invariant sets} \label{sofic}

Fix $\n>\m \geq 2$ and $\I \subseteq \Delta_{\m,\n}$. We say that a subshift $\Sigma$ of the full shift on $\I$ is \emph{sofic} if there exists a   labelled directed graph $G$  with a finite set of vertices $V$ and edges $E$, where each edge $e \in E$ has a label $\ell(e) \in \I$, such that  for each $\i \in \Sigma$, there exists an infinite path $e_1e_2 \ldots$ ($e_i \in E$) such that $\i=\ell(e_1)\ell(e_2) \ldots$. In this case we say that $G$ \emph{presents} $\Sigma$.  

Given a presentation $G$ of a sofic subshift $\Sigma$, there is a unique set of maximal irreducible subgraphs $\{G_1, \ldots, G_k\}$ of $G$, where by maximal we mean that no neighbouring vertices can be added to the subgraph while maintaining irreducibility. We call these the irreducible components of $G$. For each $1 \leq i \leq k$, define the subshift $\Sigma_{G_i}\subseteq \Sigma$ by
$$\Sigma_{G_i}=\{\ell(e_1)\ell(e_2) \ldots : e \in E_i\}$$
where $E_i$ denotes the set of edges in $G_i$. Note that $\pi \Sigma$ is a subshift which is presented by the labelled, directed graph $\pi G$, which is constructed from $G$ by projecting each label to its first coordinate. Its subgraphs $\pi G_i$ are irreducible components of $\pi G$.

Construct a labelled directed graph $H$ whose set of vertices is $\{1, \ldots, k\}$ and where there is an edge labelled $a$ from $i$ to $j$  if there is an edge labelled $a$ in $G$ from some vertex in $G_i$ to some vertex in $G_j$. Note that $H$ contains no cycles by definition of irreducible components. Then for each $1 \leq i \leq k$ we can define $\{i\}^+, \{i\}^- \subset \{1, \ldots, k\}$ by
\begin{eqnarray*}
\{i\}^+&:=& \{1 \leq j \leq k \; : \; \textnormal{there is a path in $H$ from $i$ to $j$}\}\\
 \{i\}^-&:=& \{1 \leq j \leq k \; : \; \textnormal{there is a path in $H$ from $j$ to $i$}\}
\end{eqnarray*}
noting that the definition of $\{i\}^+$ is equivalent to that provided in the introduction.
 We say that an irreducible component $G_i$ is a \emph{source} if $\{i\}^+=\{1, \ldots, k\}$ and we say that $G_i$ is a \emph{sink} if $\{i\}^-=\{1, \ldots, k\}$.

Before proving Theorem \ref{red-thm} we provide a couple of corollaries which follow from it. First, by exploiting the fact that $h(\Sigma)=\max\{h(\Sigma_{G_i})\}_{i=1}^k$ and $h(\pi\Sigma)=\max\{h(\pi\Sigma_{G_i})\}_{i=1}^k$, we can recover a simpler formula for the box dimension in the case that a source or sink has certain entropy maximising properties.

%We say that $G$ is primitive if there exists $N \in \N$ such that for all $v,w \in V$, there exists a path of length $N$ in the graph, starting at $v$ and ending at $w$.  We say that $G$ is right-resolving, if for each vertex $v \in V$, the edges starting at $v$ carry different labels. 

%We define the adjacency matrix $A$ of the graph $G$ to be the $|V|\times |V|$ matrix ($|V|$ denotes the number of vertices) defined by $(A)_{(v,w)}= \# \textnormal{paths from $v$ to $w$}$. We say $A$ is irreducible if for all $v,w \in V$ there exists $N \in \N$ such that $(A^N)_{(v,w)}>0$. We say $A$ is primitive if there exists $N \in \N$ such that for all $v,w \in V$, $(A^N)_{(v,w)}>0$. In particular, $G$ is irreducible if and only if its adjacency matrix is irreducible and $G$ is primitive if and only if its adjacency matrix is primitive.

% If $\Sigma$ is sofic, then in fact $\Sigma$ can be presented by a right-resolving labeled graph $G$. A sofic shift is transitive if and only if it can be presented by a right-resolving irreducible graph \cite[Lemma 3.3.10-11]{lm} and a sofic shift is mixing if and only if it can be presented by a primitive, right-resolving graph \cite[Ex 4.5.16]{lm}.

\begin{cor} \label{source} 
Let $G$ be a presentation of $\Sigma$ with irreducible components $\{G_1, \ldots, G_k\}$. Suppose that either:
\begin{enumerate}[(a)]
\item for some $1 \leq i\leq k$, $G_i$ is a source and $h(\Sigma_{G_i})=h(\Sigma)$ or 
\item  for some $1 \leq i\leq k$, $G_i$ is a sink and $h(\pi\Sigma_{G_i})=h(\pi\Sigma)$.
\end{enumerate}
 Then 
$$\dim_\textup{B} \Pi(\Sigma)= \frac{h(\pi\Sigma)}{\log \m} + \frac{h(\Sigma) - h(\pi\Sigma)}{\log \n} .$$
\end{cor}

Secondly, by \cite{kp-measures} we can describe which conditions guarantee (or preclude)  equality of the Hausdorff and box dimensions.

\begin{cor}\label{b=h}
The equality $\hd \Pi(\Sigma)=\bd \Pi(\Sigma)$ holds if and only if 
\begin{eqnarray}\label{b=max}
\bd \Pi(\Sigma) = \max_{1 \leq p \leq k}\left\{ \frac{h(\pi\Sigma_{G_p})}{\log \m} + \frac{h(\Sigma_{G_p}) - h(\pi\Sigma_{G_p})}{\log \n}\right\}
\end{eqnarray}
{\bf  and} the measure of maximal entropy on $\Sigma_{G_p}$ (for some $p$ which maximises the expression on the right hand side of \eqref{b=max}) projects to the measure of maximal entropy on $\pi \Sigma_{G_p}$.

In particular, if the maximum in \eqref{red} is not obtained for a pair $i=j$, that is,
\begin{eqnarray} 
\max_{1 \leq i \leq k}\left\{ \frac{h(\Sigma_{G_i})}{\log \n} + \max_{j \in \{i\}^+} h(\pi \Sigma_{G_j}) \left(\frac{1}{\log \m}-\frac{1}{\log \n}\right)\right\}>\max_{1 \leq p \leq k}\left\{ \frac{h(\pi\Sigma_{G_p})}{\log \m} + \frac{h(\Sigma_{G_p}) - h(\pi\Sigma_{G_p})}{\log \n}\right\},  \label{b>h}
\end{eqnarray}
then $\hd \Pi(\Sigma)<\bd \Pi(\Sigma)$.
\end{cor}

We will prove Corollaries \ref{source} and \ref{b=h} following the proof of Theorem \ref{red-thm} in Section \ref{red-proofs}.

\subsection{Example}

Before providing the  proofs of the results of this section, we illustrate Theorem \ref{red-thm} with an example. Put $\n=5$ and  $\m=3$. Let $\Sigma$ be the subshift of finite type presented by the graph $G$ in Figure \ref{red-fig}.

\tikzset{every loop/.style={min distance=10mm,looseness=10}}

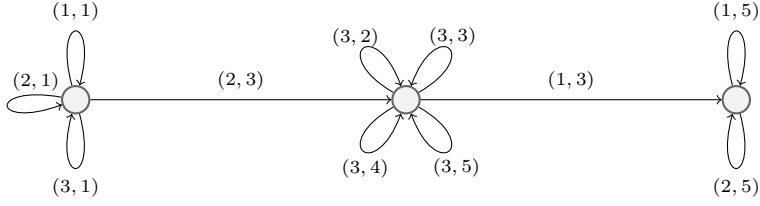
\begin{figure}[H]
\begin{center}
   \begin{tikzpicture}[
roundnode/.style={circle, draw=black!60, fill=black!5, thick, minimum size=2mm},]
        \node [roundnode] (1) {} ;
	\node[roundnode]      (2)       [right=4cm of 1] {};
	\node[roundnode]      (3)       [right=4cm of 2] {};

        \path[->] (1) edge  [in=200,out=170,loop] node[auto, xshift=-2pt, yshift=2pt] {\tiny $(2,1)$} ();
        \path[->] (1) edge  [loop above] node[auto] {\tiny $(1,1)$} ();
        \path[->] (1) edge  [loop below] node[auto] {\tiny $(3,1)$} ();

  \path[->]  (1) edge  [in=180,out=0] node[auto]  {\tiny $(2,3)$} (2);
 \path[->]  (2) edge  [in=180,out=0] node[auto]  {\tiny $(1,3)$} (3);

      \path[->] (3) edge  [loop below] node[auto] {\tiny $(2,5)$} ();
      \path[->] (3) edge  [loop above] node[auto] {\tiny $(1,5)$} ();

        \path[->] (2) edge  [in=70,out=30,loop] node[above,at start,anchor=south west, yshift=14pt] {\tiny $(3,3)$} ();
        \path[->] (2) edge  [in=-70,out=-30,loop] node[auto, xshift=-10pt, yshift=1pt] {\tiny $(3,5)$} ();
        \path[->] (2) edge  [in=110,out=150,loop] node[auto, xshift=10pt, yshift=-3pt] {\tiny $(3,2)$} ();
        \path[->] (2) edge  [in=-110,out=-150,loop] node[auto, xshift=-12pt, yshift=-14pt] {\tiny $(3,4)$} ();

    \end{tikzpicture}
\end{center} \caption{The graph $G$}  \label{red-fig}
\end{figure}
%\ref{}.\caption{The graph $G$}  \label{red-fig}
%\begin{figure}[H]  \includegraphics[width=\linewidth]{}\end{figure}

$G$ has three irreducible components $G_1, G_2, G_3$. $\Sigma_{G_1}$ is the full shift on $\{(1,1), (2,1), (3,1)\}$ and $h(\Sigma_{G_1})=h(\pi \Sigma_{G_1})=\log 3$. $\Sigma_{G_2}$ is the full shift on $\{(3,2), (3,3), (3,4), (3,5)\}$ and $h(\Sigma_{G_2})=\log 4$, $h(\pi \Sigma_{G_2})=0$. $\Sigma_{G_3}$ is the full shift on $\{(1,5), (2,5)\}$ and $h(\Sigma_{G_3})=h(\pi \Sigma_{G_3})=\log 2$.  

By Theorem \ref{red-thm},
\begin{eqnarray*}
\bd \Pi(\Sigma) &=& \max \left\{ \frac{\log 3}{\log 5} + \log 3\left(\frac{1}{\log 3}-\frac{1}{\log 5}\right), \frac{\log 4}{\log 5}+ \log 2\left(\frac{1}{\log 3}-\frac{1}{\log 5}\right), \right. \\
& & \left. \; \; \; \; \; \; \; \; \;\frac{\log2}{\log 5}+ \log 2\left(\frac{1}{\log 3}-\frac{1}{\log 5}\right)\right\} \\
&=& \max\left\{1, \log 2\left(\frac{1}{\log 5}+\frac{1}{\log 3}\right)\right\}=\log 2\left(\frac{1}{\log 5}+\frac{1}{\log 3}\right).
\end{eqnarray*}

Note that 
$$\bd  \Pi(\Sigma)< 1+\frac{\log 4- \log 3}{\log 5}=\frac{h(\pi\Sigma)}{\log 3} + \frac{h(\Sigma) - h(\pi\Sigma)}{\log 5}.$$ 
Also note that 
\begin{eqnarray*}
\bd  \Pi(\Sigma)&> &\max \left\{ \frac{\log 2}{\log 3}+\frac{\log 3- \log 2}{\log 5}, \frac{\log 4}{\log 5}, \frac{\log 2}{\log 3}\right\}\\
&=&\max_{1 \leq p \leq k}\left\{ \frac{h(\pi\Sigma_{G_p})}{\log 3} + \frac{h(\Sigma_{G_p}) - h(\pi\Sigma_{G_p})}{\log 5}\right\}.
\end{eqnarray*}

\subsection{Proofs} \label{red-proofs}

We begin by proving Theorem \ref{red-thm}. Fix a presentation $G$ of $\Sigma$ and let $1 \leq i \leq k$ and $j \in \{i\}^+$ be parameters which achieve the maximum in \eqref{red}. Roughly speaking, we show that the box dimension is exhausted by covering all regions $\Pi([\i])$ where $\i \in \Sigma_{l(\delta)}$ labels a path in $G$ which stays in the irreducible component $G_i$ for roughly $k(\delta)$ time steps before travelling to the irreducible component $G_j$ and staying inside it until time $l(\delta)$.

 Given a vertex $v$ in $G$, let $ \Sigma_n^{v+}$ denote all strings in $\Sigma_n$ which label a path beginning at $v$ and $ \Sigma_n^{v-}$  denote all strings in $\Sigma_n$ which label a path ending at $v$. For the lower bound we will require the following standard result which relates the entropy of an irreducible sofic subshift to paths in $G$. 

\begin{lma}
Let $\Sigma$ be an irreducible sofic subshift with irreducible presentation $G$. Then
\begin{equation}\lim_{n \to \infty} \frac{1}{n} \log\# \Sigma_n^{v \pm} =h(\Sigma). \label{pm} \end{equation}
\end{lma}

\begin{proof}
Since $\# \Sigma_n^{v\pm} \leq \# \Sigma_n \leq \sum_{v \in V} \# \Sigma_n^{v \pm}$ we have that $\lim_{n \to \infty} \frac{1}{n} \log \left( \max_{v \in V} \# \Sigma_n^{v \pm} \right)$ exists and equals the entropy $h(\Sigma)$. By irreducibility of $G$, there exists $M \in \N$ such that for $n>M$, $\# \Sigma_n^{v\pm} \geq \max_{w \in V} \# \Sigma_{n-M}^{w \pm}$. Hence
$$\liminf_{n \to \infty} \frac{1}{n} \log \# \Sigma_n^{v \pm} \geq \lim_{n \to \infty} \frac{1}{n} \log \left( \max_{v \in V} \# \Sigma_n^{v \pm} \right)= h(\Sigma),$$
completing the proof of (\ref{pm}).
\end{proof}

\begin{proof}[Proof of lower bound in Theorem \ref{red-thm}]
Fix $\varepsilon>0$, $\delta>0$. Let $1 \leq i \leq k$ and $j \in \{i\}^+$ be the indices that maximise the expression in (\ref{red}). Let $v$ be a vertex in $G_i$ and $w$ be a vertex in $G_j$. Since $j \in \{i\}^+$, there exists a path of some length $N$ in $G$ from $v$ to $w$ which is labelled by $\k \in \Sigma_N$.  We may assume $\delta$ is small enough to ensure $k(\delta)>N$.  Given $\i \in\Sigma_{G_i, k(\delta)-N}^{v-}$, 
$$\bigcup_{ \j \in  \Sigma_{G_j, l(\delta)-k(\delta)}^{w+}}\Pi([\i\k\j]) \subseteq \Pi([\i]).$$
Therefore,
$$N_\delta(\Pi[\i]) \gtrsim N_\delta\left(\bigcup_{ \j \in  \Sigma_{G_j, l(\delta)-k(\delta)}^{w+}}\Pi([\i\k\j])\right) \approx\# \pi \Sigma_{G_j, l(\delta)-k(\delta)}^{w+}.  $$
Hence
$$N_\delta(\Pi(\Sigma)) \gtrsim \sum_{\i \in\Sigma_{G_i, k(\delta)-N}^{v-}} N_\delta([\i]) \gtrsim \#\Sigma_{G_i, k(\delta)-N}^{v-} \cdot \# \pi \Sigma_{G_j, l(\delta)-k(\delta)}^{w+} .$$
Therefore by \eqref{pm},
$$N_\delta(\Pi(\Sigma)) \gtrsim _\varepsilon \exp\left((h(\Sigma_{G_i})-\varepsilon)(k(\delta)-N)\right) \exp\left((h(\pi\Sigma_{G_j})-\varepsilon)(l(\delta)-k(\delta))\right)$$
and by letting $\delta \to 0$ we obtain the desired lower bound since $\varepsilon>0$ was arbitrary.
\end{proof}

For the upper bound we will require the standard result that the entropy of a sofic subshift equals the maximum entropy of its irreducible subshifts. 

\begin{lma}\label{red-entropy}
Let $\Sigma$ be a sofic subshift presented by a graph $G$ which has irreducible components $G_1, \ldots, G_k$. Then $h(\Sigma)=\max_{1 \leq i \leq k} h(\Sigma_{G_i})$. 
\end{lma}

\begin{proof} It is only necessary to prove the upper bound which follows by bounding $\# \Sigma_n$ above by
$$\sum_{1 \leq m_1, \ldots, m_l \leq k} \;\sum_{n_{m_1}+ \cdots + n_{m_l}=n} \# \Sigma_{G_{m_1},n_{m_1}} \cdot \# \Sigma_{G_{m_2},n_{m_2}} \cdots \# \Sigma_{G_{m_l},n_{m_l}},$$
where $\Sigma_{G_i, n}$ denotes all distinct strings of length $n$ that appear in the subgraph $G_i$, and bounding $ \# \Sigma_{G_{m_i},n_{m_i}}$ in terms of $ h(\Sigma_{G_{m_i}})$.
\end{proof}

%%%%%%%
\begin{comment}
\begin{proof}
Clearly  $h(\Sigma)\geq\max_{1 \leq i \leq k} h(\Sigma_{G_i})$ so we only need to prove the upper bound. Let $m_1, \ldots, m_l \in \{1, \ldots  , k\}$ such that for each $1 \leq i \leq l-1$, there is an edge in $H$ from $G_{m_i}$ to $G_{m_{i+1}}$. Note that there are only finitely many distinct sequences $m_1, \ldots , m_l$ with this property, and we denote them by $\mathcal{P}$. Also note that for any such sequence, \textcolor{red}{$l \leq k$.}

Let $\Sigma_{G_i, n}$ denote all distinct strings of length $n$ that appear in the subgraph $G_i$. \textcolor{red}{Given $\epsilon>0$, we can choose $K \geq 1$ such that for all $i$ and $n$
\[
\# \Sigma_{G_{i},n} \leq  \max\{e^{n (h(\Sigma_{G_i})+\epsilon)}, K\}
\]
So,
\begin{eqnarray*}
\# \Sigma_n &\leq& \sum_{m_1, \ldots, m_l \in \mathcal{P}}\sum_{n_{m_1}+ \cdots + n_{m_l}=n} \# \Sigma_{G_{m_1},n_{m_1}} \cdot \# \Sigma_{G_{m_2},n_{m_2}} \cdots \# \Sigma_{G_{m_l},n_{m_l}} \\
&\leq &  \sum_{m_1, \ldots, m_l \in \mathcal{P}}\sum_{n_{m_1}+ \cdots + n_{m_l}=n} \prod_{i=1}^l \max\{ e^{n_{m_i} (h(\Sigma_{G_{m_i}})+\epsilon)}, K\} \\
&\leq& |\mathcal{P}|{n +l \choose l}\max_{1 \leq i \leq k}  e^{n (h(\Sigma_{G_{i}})+\epsilon)}K^l\\
& \leq& (n+k)^k K^k \max_{1 \leq i \leq k}  e^{n (h(\Sigma_{G_{i}})+\epsilon)}
\end{eqnarray*}
completing the proof of the upper bound. I tried to make rigorous - hope I didn't screw up.}
\end{proof} 
\end{comment}
%%%%%%

Let $\Sigma_{n}^{G_i+} $ denote all words in $\Sigma_n$ which label paths in $G$ that start at  any vertex in $G_i$, and $\Sigma_{n}^{G_i-} $ denote all words in $\Sigma_n$ which label paths in $G$ that end at  any vertex in $G_i$. 
%Similarly let $\#\pi \Sigma_{n}^{G_i\pm} $ be the number of words in $\pi \Sigma_n$ which label paths in $\pi G$ which start/end (respectively) at any vertex in $\pi G_i$.

\begin{proof}[Proof of upper bound in Theorem \ref{red-thm}]
Fix $\varepsilon>0$, $\delta>0$.  Fix any $1 \leq i \leq k$ and $\i \in \Sigma_n^{G_i-}$. Writing $ M(\i, l) = \# \pi(\j \in \Sigma_l : \j|_{k(\delta)}=\i)$ and noting that by shift invariance we have $M(\i,l(\delta)) \leq \# \pi \Sigma^{G_i+}_{l(\delta)-k(\delta)}$ it follows that
$$N_\delta(\Pi(\Sigma)) \lesssim  \sum_{i=1}^k \sum_{\i \in \Sigma_{k(\delta)}^{G_i-}} N_\delta(\Pi([\i])) \approx  \sum_{i=1}^k \sum_{\i \in   \Sigma_{k(\delta)}^{G_i-}} M(\i, l(\delta)) \lesssim \# \Sigma_{k(\delta)}^{G_i-} \# \pi \Sigma_{l(\delta)-k(\delta)}^{G_i+}.$$
Note that any path that ends at a vertex in $G_i$ is contained in the minimal subgraph $E$ of $G$ which contains the irreducible components $\{G_j\}_{j \in \{i\}^-}$ and all edges between these components. Similarly, any path in $G$ that begins at a vertex in $G_i$ is contained in the minimal subgraph $ F$ of $G$ where $F$ contains the irreducible components $\{ G_j\}_{j \in \{i\}^+}$, and all edges between these components. By Lemma \ref{red-entropy},  
$$\limsup_{n \to \infty} \frac{1}{n} \log\# \Sigma_{n}^{G_i-} \leq h(\Sigma_{E})=\max_{j \in \{i\}^-} h(\Sigma_{G_j})$$
 and 
$$\limsup_{n \to \infty} \frac{1}{n} \log\# \pi \Sigma_{n}^{G_i+} \leq h(\pi \Sigma_F)=\max_{j \in \{i\}^+}h(\pi \Sigma_{G_j}).$$ 
Therefore,
\begin{eqnarray*}
N_\delta(\Pi(\Sigma)) \lesssim_\varepsilon \exp\left(k(\delta)(\max_{j \in \{i\}^-} h(\Sigma_{G_j})+\varepsilon)\right) \exp\left((l(\delta)-k(\delta))(\max_{j \in \{i\}^+}h(\pi \Sigma_{G_j})+\varepsilon)\right),
\end{eqnarray*}
and by letting $\delta \to 0$ we obtain the desired upper bound since $\varepsilon>0$ was arbitrary.
\end{proof}

\begin{proof}[Proof of Corollary \ref{source} ] 
First, to see (a), let $G_i$ be the source. By assumption $h(\Sigma_{G_i})=h(\Sigma)$. By Lemma \ref{red-entropy} there exists $1 \leq j \leq k$ such that $h(\pi \Sigma_{G_j})=h(\pi \Sigma)$. Moreover $j \in \{i\}^+$ by definition of a source. Hence by (\ref{red}), 
$$ \dim_\textup{B} \Pi(\Sigma) \geq \frac{h(\Sigma_{G_i})}{\log \n} +  h(\pi \Sigma_{G_j}) \left(\frac{1}{\log \m}-\frac{1}{\log \n}\right)=\frac{h(\pi \Sigma)}{\log \m}+\frac{h(\Sigma)-h(\pi\Sigma)}{\log \n}.$$
On the other hand, the upper bound follows from Proposition \ref{bounds}, completing the proof of (a).

Similarly, to see (b), let $G_j$ be the sink. By assumption $h(\pi \Sigma_{G_j})=h(\pi \Sigma)$.  Also, by Lemma \ref{red-entropy} there exists $1 \leq i \leq k$ such that $h(\Sigma_{G_i})=h(\Sigma)$. Moreover $i \in \{j\}^-$ by definition of a sink. Hence by (\ref{red}),
$$ \dim_\textup{B} \Pi(\Sigma) \geq \frac{h(\Sigma_{G_i})}{\log \n} +  h(\pi \Sigma_{G_j}) \left(\frac{1}{\log \m}-\frac{1}{\log \n}\right)=\frac{h(\pi \Sigma)}{\log \m}+\frac{h(\Sigma)-h(\pi\Sigma)}{\log \n}.$$
The upper bound follows from Proposition \ref{bounds}, completing the proof of (b).
\end{proof}

\begin{proof}[Proof of Corollary \ref{b=h} ] 
First we recall that by \cite{kp-measures}, any ergodic invariant measure $\mu$ on $\Sigma$ satisfies the Ledrappier-Young formula:
\begin{equation}
 \hd \mu = \frac{h(\pi \mu)}{\log \m} + \frac{h(\mu)-h(\pi \mu)}{\log \n},  \label{ly}
\end{equation}
where $\hd \mu$ denotes the Hausdorff dimension of $\mu$, $h(\mu)$ denotes the measure-theoretic entropy of $\mu$ with respect to the left shift map on $\Sigma$ and $h(\pi \mu)$ denotes the measure-theoretic entropy of the pushforward measure $\pi \mu$ with respect to the left shift on $\pi \Sigma$. 

First, suppose the equality \eqref{b=max} holds and let $1 \leq p \leq k$ be an index that maximises the right hand side of \eqref{b=max}. Let $\mu$ be the ergodic invariant measure which maximises entropy on $\Sigma_{G_p}$, which we will assume projects to the measure which maximises entropy on  $\pi\Sigma_{G_p}$.  Then by \eqref{ly},
\[
\hd \Pi(\Sigma) \geq \hd \mu = \frac{h(\pi \mu)}{\log \m} + \frac{h(\mu)-h(\pi \mu)}{\log \n}  = \bd \Pi(\Sigma) 
\]
by \eqref{b=max}.  

For the converse, we assume that $\bd \Pi(\Sigma)=\hd \Pi(\Sigma)$. By \cite{kp-measures} there exists an ergodic invariant measure $\mu$ of maximal Hausdorff dimension. \footnote{The statement of \cite[Theorem 1.1]{kp-measures} does not make explicit that a measure of maximal Hausdorff dimension can be taken to be ergodic, however this is clear from its proof.}  Since $\mu$  is ergodic, its support must be contained in $\Sigma_{G_i}$ for some irreducible component $G_i$ of $G$.  Therefore, using \eqref{ly} we obtain
\begin{eqnarray*}
\hd \Pi(\Sigma) = \hd \mu &=&   \frac{h(\pi \mu)}{\log \m} + \frac{h(\mu)-h(\pi \mu)}{\log \n} \\ 
 &\leq & \max_{1 \leq p \leq k}\left\{ \frac{h(\pi \Sigma_{G_p})}{\log \m} + \frac{h(\Sigma_{G_p})-h(\pi \Sigma_{G_p})}{\log \n}\right\} \\ 
&\leq& \bd \Pi(\Sigma).
\end{eqnarray*}
Now, if \eqref{b=max} does not hold, then the second inequality above is strict and thus we get a contradiction. On the other hand, if \eqref{b=max} holds but the measure of maximal entropy $\mu_p$ on $\Sigma_{G_p}$ does \emph{not} project to the measure of maximal entropy on $\pi \Sigma_{G_p}$ for any $1 \leq p \leq k$ that maximises the right hand side of \eqref{b=max}, then the first inequality above is strict yielding a contradiction and completing the proof.
\end{proof}

\section{Coded subshifts} \label{coded}

Fix $\n>\m \geq 2$ and $\I \subseteq \Delta_{\m,\n}$. Let $\C=\{c_i\}_{i=1}^\infty$ be a countable family of words on the alphabet $\I$. We call $\C$ the generators. Let $\C_n:=\C \cap \I^n$. Define
$$B:=\{sc_{i_1}c_{i_2} \ldots \; : \; c_{i_j} \in \C, \; \textnormal{$s$ is a suffix of a word in $\C$}\}.$$
Note that $B$ is $\sigma$-invariant but may not be compact. We define $\Sigma=\overline{B}$ and say that $\Sigma$ is a coded subshift. Note that $\pi \Sigma$ is also a coded subshift which is generated by $\pi \C$. Recall that we say that the coded subshift $\Sigma$ satisfies \emph{unique decomposition with respect to $\C$} if no finite word in $\Sigma^*$ can be written by concatenating generators in $\C$ in distinct ways. Note that if $\Sigma$ satisfies unique decomposition with respect to $\C$, this does not necessarily mean that $\pi \Sigma$  satisfies unique decomposition with respect to $\pi\C$, although it may satisfy unique decomposition with respect to a different generating set (for instance if $\Sigma$ satisfies unique decomposition with respect to $\C$ and $\{(1,2), (1,3), (1,4)(1,4)\} \subset \C$ then since $\{1, 11\} \subset \pi \C$, $\pi \Sigma$ does not satisfy unique decomposition with respect to $\pi \C$). 

Construct a directed labelled graph by fixing a vertex $v$ and, for each $i \in \N$, adding a path which begins and ends at $v$ which is labelled by the generator $c_i$, such that the paths do not intersect each other apart from at the start and end points. We call these generating loops. We say that $G$ presents the coded subshift $\Sigma$.\footnote{Note that this notion of the presentation of a coded subshift differs from the notion of the presentation of a sofic subshift. If $\Sigma$ is sofic then all infinite sequences in $\Sigma$ label an infinite path in its presentation, whereas if $\Sigma$ is coded then this is need not be the case (i.e. if $\Sigma \setminus B \neq \emptyset$).} Similarly, construct the graph $\pi G$ from $G$ by projecting each label to its first coordinate and removing any generating loop which bears the same sequence of labels as another generating loop (so that each generating loop is labelled uniquely by a generator in $\pi \C$). Then $\pi G$ presents the coded subshift $\pi \Sigma$.

Let $\G_n$ denote all words in $\Sigma_n$ which label a path in $G$ that begins and ends at the vertex $v$, and $\G=\bigcup_{n=1}^\infty \G_n$. In particular, $\G$ consists of concatenations of generators. Analogously, $\pi \G$ are all words in $\pi \Sigma$ which label a path in $\pi G$ that begins and ends at the vertex $v$. We denote
\begin{eqnarray*}
h:=\limsup_{n \to \infty} \frac{1}{n} \log \# \G_n \quad \quad & \textnormal{and} & \quad \quad  h_\pi:=\limsup_{n \to \infty} \frac{1}{n} \log \# \pi\G_n.
\end{eqnarray*}
Note that the $\limsup$ is necessary in the definitions above, for instance consider a coded subshift generated by a set of generators which all have even length. Also, note that these definitions are equivalent to those recorded in the introduction. 

We begin by proving Theorem \ref{g-thm}, namely that if $h=h(\Sigma)$ and $h_\pi=h(\pi \Sigma)$ then $\bd \Pi(\Sigma)$ equals its trivial upper bound.

\begin{proof}[Proof of Theorem \ref{g-thm}]
By Proposition \ref{bounds} it suffices to prove the lower bound. Fix $\varepsilon>0$. Since $h=h(\Sigma)$ and $h_\pi=h(\pi \Sigma)$ we can choose $m_\varepsilon, n_\varepsilon \in \N$ such that 
$$\# \G_{m_\varepsilon} \geq e^{m_\varepsilon (h(\Sigma)-\varepsilon)} \;\;\; \textnormal{and} \;\;\;\# \pi\G_{n_\varepsilon} \geq e^{n_\varepsilon( h(\pi\Sigma)-\varepsilon)}.$$
In particular, for all $k \in \N$,
$$\# \G_{km_\varepsilon} \geq e^{km_\varepsilon (h(\Sigma)-\varepsilon)} \;\;\; \textnormal{and} \;\;\;\# \pi\G_{kn_\varepsilon} \geq e^{kn_\varepsilon (h(\pi\Sigma)-\varepsilon)}$$
since $\#\G_{kn} \geq (\# \G_n)^k$ and $\#\pi\G_{kn} \geq (\#\pi \G_n)^k$.
Let $\delta>0$ be sufficiently small that $k(\delta) \geq 2m_\varepsilon$ and $l(\delta)-k(\delta) \geq 2n_\varepsilon$. Hence we can find $k(\delta)-m_\varepsilon<k'(\delta) \leq k(\delta)$ which is a multiple of $m_\varepsilon$, that is,
$$\#\G_{k'(\delta)} \geq e^{k'(\delta)(h(\Sigma)-\varepsilon)}.$$
Similarly we can find $l(\delta)-n_\varepsilon<l'(\delta) \leq l(\delta)$ such that $l'(\delta)-k'(\delta)$ is a multiple of $n_\varepsilon$, that is,
$$\#\pi\G_{l'(\delta)-k'(\delta)} \geq e^{(l'(\delta)-k'(\delta))(h(\pi\Sigma)-\varepsilon)}.$$
  Denoting $M(\i,l)=\# \pi(\j \in \Sigma_l : \j|_{k'(\delta)}=\i)$, we have
\begin{eqnarray*}
N_\delta(\Pi(\Sigma)) \gtrsim \sum_{\i \in \G_{k'(\delta)}} N_\delta(\Pi([\i]))& \gtrsim& \sum_{\i \in \G_{k'(\delta)}} M(\i,l'(\delta))\\
&\geq & \# \G_{k'(\delta)} \# \pi\G_{l'(\delta)-k'(\delta)} \\
&\geq& e^{k'(\delta)(h(\Sigma)-\varepsilon)}  e^{(l'(\delta)-k'(\delta))(h(\pi\Sigma)-\varepsilon)} \\
&\gtrsim_\varepsilon&  e^{k(\delta)(h(\Sigma)-\varepsilon)}  e^{(l(\delta)-k(\delta))(h(\pi\Sigma)-\varepsilon)} .
\end{eqnarray*}
The lower bound follows since $\varepsilon$ was chosen arbitrarily.
\end{proof}

Conversely, examples can be constructed where either $h< h(\Sigma)$ or $h_\pi < h(\pi \Sigma)$ and the conclusion of Theorem \ref{g-thm} does \emph{not} hold, that is, the dimension $\bd \Pi(\Sigma)$ drops from the trivial upper bound. In particular, in \S \ref{sdimdrop} we will construct an example where $h< h(\Sigma)$ and $\bd \Pi(\Sigma)$ equals the trivial lower bound $\max\left\{\frac{h(\Sigma)}{\log \n}, \frac{h(\pi \Sigma)}{\log \m}\right\}$ thereby settling Theorem \ref{dimdrop}.

The drawback of Theorem \ref{fgeq1} is that generally it is not straightforward to verify the equalities $h=h(\Sigma)$ and $h_\pi=h(\pi\Sigma)$. However, under the assumption of unique decomposition of $\Sigma$ and $\pi \Sigma$ we can provide more checkable conditions that guarantee the box dimension $\bd \Pi(\Sigma)$ to equal its trivial upper bound (Theorem \ref{fgeq1}).

\subsection{Coded subshifts with unique decomposition}

Throughout this short section we will assume that $\Sigma$ is a coded subshift with unique decomposition with respect to $\C$ and that the coded subshift $\pi \Sigma$ satisfies unique decomposition with respect to $\pi \C$. Let $G$ and $\pi G$ be the presentations of $\Sigma$ and $\pi \Sigma$ as detailed in the previous section. Let $p_G(v,n)$ denote the number of paths of length $n$ in $G$ which begin and end at $v$ and $p_{\pi G}(v,n)$ denote the number of paths of length $n$ in $\pi G$ which begin and end at $v$ and write
\begin{eqnarray*}
h_G:=\limsup_{n \to \infty} \frac{1}{n} \log p_G(v,n) \quad \quad & \textnormal{and} & \quad \quad  h_{\pi G}:=\limsup_{n \to \infty} \frac{1}{n} \log \# p_{\pi G}(v,n).
\end{eqnarray*}

In particular, $h_G$ is the Gurevic entropy of $G$ and $h_{\pi G}$ is the Gurevic entropy of $\pi G$, noting that the limsups are actually independent of the choice of vertex. Since $\Sigma$ and $\pi \Sigma$ satisfy unique decomposition with respect to $\C$ and $\pi \C$ respectively, we have $h=h_G$ and $h_\pi=h_{\pi G}$. This will enable us to apply techniques from the theory of countable Markov shifts.

Recall from the introduction the functions $f, f_\pi: [0, \infty) \to (0,\infty]$ which we defined by 
\begin{eqnarray} \label{f2}
f(x)= \sum_{n=1}^\infty \#\C_n e^{-nx} \; \; \; \; & \textnormal{and} & \; \; \; \; f_{\pi}(x)= \sum_{n=1}^\infty \#\pi \C_n e^{-nx}.
\end{eqnarray}

We can apply the classical work of Vere-Jones \cite{vj} to deduce behaviour of $f$ and $f_\pi$ at $h$ and $h_\pi$.

\begin{lma}
Let $\Sigma$ and $\pi \Sigma$ be coded subshifts with unique decomposition with respect to generating sets $\C$ and $\pi \C$ respectively. Then
\begin{eqnarray} \label{vj-eqn}
f(h)\leq 1 \; \; \; \; & \textnormal{and} & \; \; \; \; f_{\pi}(h_\pi)\leq 1 .
\end{eqnarray}
\end{lma}

\begin{proof}
Let $q_G(v,n)$ denote the number of generating loops of length $n$ in $G$. Let $q_{\pi G}(v,n)$ denote the number of generating loops of length $n$ in $\pi G$. In particular, $q_G(v,n)=\#\C_n$ and $q_{\pi G}(v,n)=\# \pi \C_n$, so $f(x)=\sum_{n=1}^\infty q_G(v,n) e^{-nx}$ and $f_\pi(x)=\sum_{n=1}^\infty q_{\pi G}(v,n) e^{-nx}$. By using the recurrence relation $p(v,n)=\sum_{k=1}^n p_G(v, n-k) q_G(v,k)$ and an application of a renewal theorem, Vere-Jones \cite[Lemma 2]{vj} showed that $\sum_{n=1}^\infty q_G(v,n)e^{-nh_G } \leq 1$, and analogously $\sum_{n=1}^\infty q_{\pi G}(v,n)e^{-nh_{\pi G} } \leq 1$. This implies the result since  $h=h_G$ and $h_\pi=h_{\pi G}$ by unique decomposition.
\end{proof}

Next recall the definitions from the introduction
\begin{eqnarray*}
\ell:=\limsup_{n \to \infty} \frac{1}{n} \log \# \L_n \quad \quad & \textnormal{and} & \quad \quad  \ell_\pi:=\limsup_{n \to \infty} \frac{1}{n} \log \# \pi\L_n
\end{eqnarray*}
where $\L_n$ and $\pi \L_n$ denote words of length $n$ which appear at the beginning or end of generators in $\C$ and $\pi \C$ respectively. In \cite{ct} it was shown that $\ell<h$ implies existence of a measure of maximal entropy for the coded subshift $\Sigma$. The behaviour of $f$ at a quantity related to $\ell$ was used in \cite{pavlov} to characterise coded subshifts in terms of the properties of their measures of maximal entropy.

To prove Theorem \ref{fgeq1} we will show that $f(\ell) > 1$ implies $\ell < h$ by using \eqref{vj-eqn} and the fact that $f$ is strictly decreasing, and then by naturally decomposing words in $\Sigma$ into concatenations of generators and subwords of generators we will deduce that this implies $h=h(\Sigma)$ (respectively $h_\pi=h(\pi \Sigma)$).

\begin{lma} \label{g lemma}
Suppose $\Sigma$ is a coded subshift which satisfies unique decomposition with respect to a generating set $\C$ and $f(\ell) > 1$. Then 
$$\limsup_{n \to \infty} \frac{1}{n} \log \#\G_n= h(\Sigma).$$
\end{lma}

\begin{proof}
Assume that $f(\ell) >1$. Since $h_G=h$ by unique decomposition it follows that $f(h)=f(h_G) \leq 1$ by \eqref{vj-eqn} and therefore since $f$ is strictly decreasing we have $\ell<h_G=h \leq h(\Sigma)$ (the second inequality follows trivially from the definition of $\G$). We will show that $\ell < h(\Sigma)$ implies that $h=h(\Sigma)$, using arguments similar to those contained in  \cite[\S 5.1]{ct}.

Let $\varepsilon>0$ be sufficiently small that $\ell-h(\Sigma)+2\varepsilon<0$. 

Then
\begin{eqnarray}
\# \Sigma_n &\leq& \sum_{i+j+k=n} \# \L_i \# \G_j \# \L_k \nonumber \\
&\lesssim_\varepsilon& \sum_{i+j+k=n} e^{(i+k)(\ell+\varepsilon)}\# \G_j \nonumber \\
&=&  \sum_{j=0}^n (n-j)e^{(n-j)(\ell+\varepsilon)} \# \G_j. \label{decomp}
\end{eqnarray}
Hence
\begin{eqnarray*}
\sum_{j=0}^n e^{(n-j)(\ell-h(\Sigma)+2\varepsilon)}(n-j)\frac{\# \G_j}{\# \Sigma_j} \gtrsim_\varepsilon\sum_{j=0}^n e^{(n-j)(\ell+\varepsilon)}(n-j)\frac{\# \G_j}{\# \Sigma_j}  \frac{\#\Sigma_j}{\# \Sigma_n} \gtrsim_\varepsilon 1.
\end{eqnarray*}
In particular, there exists $c_\varepsilon>0$ such that
\begin{eqnarray}
\sum_{j=0}^n e^{(n-j)(\ell-h(\Sigma)+2\varepsilon)}(n-j)\frac{\# \G_j}{\# \Sigma_j} \geq c_\varepsilon. \label{g eqn}
\end{eqnarray}
Since $\ell-h(\Sigma)+2\varepsilon<0$ we  can choose $N \in \N$ sufficiently large that
\begin{eqnarray*}
\sum_{j=0}^{n-N} e^{(n-j)(\ell-h(\Sigma)+2\varepsilon)}(n-j)\frac{\# \G_j}{\# \Sigma_j}& \leq& \sum_{j=0}^{n-N}e^{(n-j)(\ell-h(\Sigma)+2\varepsilon)}(n-j) \\
&\leq& \sum_{m \geq N} e^{m(\ell-h(\Sigma)+2\varepsilon)}m \leq \frac{c_\varepsilon}{2}.
\end{eqnarray*}

 Hence by (\ref{g eqn}) 
\begin{eqnarray}
\sum_{j=n-N+1}^n e^{(n-j)(\ell-h(\Sigma)+2\varepsilon)}(n-j)\frac{\# \G_j}{\# \Sigma_j} \geq \frac{c_\varepsilon}{2}. \label{g eqn2} \end{eqnarray}
If we let $C$ be a uniform upper bound on $e^{m(\ell-h(\Sigma)+2\varepsilon)}m$ (for $m \geq 0$), we can deduce from (\ref{g eqn2}) that
$$\sum_{j=n-N+1}^n \frac{\# \G_j}{\# \Sigma_j} \geq \frac{c_\varepsilon}{2C}$$
hence for all $n \geq N+1$ we have $\# \G_j \geq \frac{c_\varepsilon}{2CN} \# \Sigma_j$ for some $n-N+1 \leq j \leq n$. This implies that $h= \limsup_{n \to \infty} \frac{1}{n} \log \# \G_n=h(\Sigma)$. 
\end{proof}

Clearly by combining Lemma \ref{g lemma} with Theorem \ref{g-thm} we establish Theorem \ref{fgeq1}: that if $\Sigma$ and $\pi \Sigma$ satisfy unique decomposition with respect to $\C$ and $\pi \C$ and we have that $f(\ell)> 1$ and $f_{\pi}(\ell_\pi) >  1$ then
\begin{eqnarray} \bd \Pi(\Sigma)=\frac{h(\pi\Sigma)}{\log \m} + \frac{h(\Sigma) - h(\pi\Sigma)}{\log \n}.\label{usual2} \end{eqnarray}
Hence to establish \eqref{usual2} for uniquely decomposing coded subshifts $\Sigma$ and $\pi \Sigma$, it is sufficient to calculate $f(\ell)$ and $f_\pi(\ell_\pi)$, which solely depend on $\#\C_n, \#\pi \C_n, \# \L_n$ and $\# \pi \L_n$ which are often easy to compute. We demonstrate this with some examples in the next section.

\subsection{Examples} In this section, we illustrate Theorems \ref{dimdrop}, \ref{g-thm} and \ref{fgeq1} with some examples. First, in  \S \ref{beta} we describe how Theorem \ref{g-thm} can be applied to $\beta$-shifts. In \S \ref{good} we apply Theorem \ref{fgeq1} to (generalised) $S$-gap shifts. Finally in \S \ref{sdimdrop} we construct an example of a coded subshift $\Sigma$ where $h<h(\Sigma)$ and $h_\pi=h(\pi\Sigma)$ and
 $$\bd \Pi(\Sigma)= \max\left\{\frac{h(\Sigma)}{\log \n}, \frac{h(\pi \Sigma)}{\log \m}\right\}<\frac{h(\pi\Sigma)}{\log \m} + \frac{h(\Sigma) - h(\pi\Sigma)}{\log \n}$$
thereby proving Theorem \ref{dimdrop}.  

%%%%%%%
\begin{comment}
\vspace{2mm}

\begin{tabular}{| c | c | c |}
\hline 
 & $f(h(L)) \geq 1$ & $ f(h(L)) < 1$ \\
\hline
$f_{\pi }(h(L_\pi))\geq 1$ & Theorem \ref{fgeq1} holds & Counterexample from previous section\\
\hline
$ f_{\pi }(h(L_\pi))< 1$ & Example \ref{><} is a counterexample? & Example \ref{<<} is a counterexample? \\
\hline
\end{tabular}

\vspace{2mm}
\end{comment}
%%%%%%%

\subsubsection{$\beta$-shifts} \label{beta}

Fix $\n>\m$ and $\I \subset \Delta_{\m,\n}$. We begin by describing a subshift on the set of digits $\I$ which is conjugate to the $\beta$-shift on the usual digit set $\{0, \ldots, \lfloor \beta \rfloor\}$, for more details see \cite{blanchard} or \cite{ct} and references therein.

 Fix a bijection $\O:\{0,\ldots, |\I|-1\} \to \I$ which will determine an ordering on the elements in $\I$. We extend $\O$ to finite and infinite words with digits in $\{0, \ldots, |\I|-1\}$ by $\O(i_1i_1 \ldots)=\O(i_1)\O(i_2)\ldots$. Fix $|\I|<\beta<|\I|+1$ and let $(b_n)_{n \in \N}$ be the greedy $\beta$-expansion of 1, meaning the lexicographically maximal solution to
$$\sum_{n=1}^\infty b_n \beta^{-n}=1.$$
We define
$$\Sigma=\left\{ \O((x_n)_{n \in \N}) \; : \; (x_n)_{n \in \N} \in \{0, \ldots, |\I|-1\}^\N \; \textnormal{s.t.} \; \sigma^k((x_n)_{n \in \N}) \preceq (b_n)_{n \in \N} \; \forall k \in \N\right\}$$
where $\preceq$ stands for the lexicographic order. In particular, $\Sigma$ is conjugated by $\O$ to the $\beta$-shift on the set of digits $\{0, \ldots, \lfloor \beta \rfloor\}$. Therefore it is known \cite{blanchard} that $\Sigma$ is a coded subshift where the set of generators is given by
$$\C=\bigcup_{n \geq 1: b_n>0} \{\O(b_1 \ldots b_{n-1}0), \ldots, \O(b_1 \ldots b_{n-1} (b_n-1))\}.$$

Note that any word in $\i \in \Sigma^*$ can be written $\i=c_1 \ldots c_k w$ where $c_i \in \C$ and $w$ is a word that appears at the beginning of a generator in $\C$. Hence
\begin{eqnarray}
\#\Sigma_n \leq \sum_{k=1}^n \# \G_{n-k} \label{beta1}
\end{eqnarray}
since for each $k \in \N$, $\O(b_1 \ldots b_k)$ is the unique word of length $k$ that appears at the beginning of a generator in $\C$. Similarly, we have
\begin{eqnarray}
\#\pi\Sigma_n \leq \sum_{k=1}^n \# \pi\G_{n-k}. \label{beta2}
\end{eqnarray}
Using \eqref{beta1} and \eqref{beta2} it is easy to adapt the set of inequalities \eqref{decomp} and the estimates that follow it to deduce that $h=h(\Sigma)$ and $h_\pi=h(\pi \Sigma)$. In particular $\bd \Pi(\Sigma)=\frac{h(\pi\Sigma)}{\log \m} + \frac{h(\Sigma) - h(\pi\Sigma)}{\log \n}$ by Theorem \ref{g-thm}.

\subsubsection{Generalised $S$-gap shifts} \label{good}

We begin by considering the following natural generalisation of the $S$-gap shifts \cite{lm,ct}. Fix any $\n>\m$, $\I\subset \Delta_{\m,\n}$ and $(i_0,j_0) \in \I$ such that $\pi(\I \setminus \{(i_0,j_0)\}) \cap \{i_0\}=\emptyset$. Fix a countable set $S \subset \N$. Put
$$\C=\{w\,(i_0,j_0) \; : \; w \in (\I \setminus \{(i_0,j_0)\})^s, \; s \in S\}.$$
We consider the coded subshift $\Sigma$ generated by $\C$. Under the assumptions on $\I$, both $\Sigma$ and $\pi \Sigma$ satisfy unique decomposition with respect to $\C$ and $\pi \C$ respectively. The classical $S$-gap shifts correspond to the case that $\# \I=2$, however since analysis of the box dimension of $\Pi(\Sigma)$ is trivial for subshifts on 2 symbols we are primarily interested in the case that $\# \I \geq 3$.

Observe that
\begin{eqnarray*}
\# \C_n=
\begin{cases} 
(\# \I-1)^{n-1} &  n \in S \\
0 &   n \notin S.
\end{cases}
\end{eqnarray*}
Also clearly $\ell=\log (\# \I-1)$. Therefore,
$$f(\ell)=\frac{(\# \I-1)^{n-1}}{(\# \I-1)^n}=\sum_{n \in S} \frac{1}{\# \I-1}=\infty >1.$$
Similarly we can calculate that
\begin{eqnarray*}
\# \pi \C_n=
\begin{cases} 
(\# \pi\I-1)^{n-1} &  n \in S \\
0 &   n \notin S.
\end{cases}
\end{eqnarray*}
and $\ell_\pi=\log (\#\pi \I-1)$, so $f_\pi(\ell_\pi)=\infty$. In particular, Theorem \ref{fgeq1} is applicable and we deduce that $\bd \Pi(\Sigma)=\frac{h(\pi\Sigma)}{\log \m} + \frac{h(\Sigma) - h(\pi\Sigma)}{\log \n}$.

%%%%%%%%%%% Possible additional example %%%%%%%%%%%%%
\begin{comment}
\subsubsection{Another example where Theorem \ref{fgeq1} is applied?} 

  Fix any $\n>\m \geq 2$, $1 \leq i, j \leq m$ where $i \neq j$ and put $\I=\{(i,1), (i,2), (j,1), (j,2)\}$. Put
$$C=\{w\,(j,1)^{|w|}, w\,(j,2)^{|w|} \; : \; w \in \{(i,1), (i,2)\}^*\}.$$
We consider the coded subshift $\Sigma$ generated by $\C$. Observe that
\begin{eqnarray*}
\# \C_n=
\begin{cases} 
2^{n+1} &  n \; \textnormal{even} \\
0 &  n\;  \textnormal{odd}
\end{cases}
\end{eqnarray*}
and
\begin{eqnarray*}
\# \pi \C_n=
\begin{cases} 
1 &  n \; \textnormal{even} \\
0 &   n \;\textnormal{odd}.
\end{cases}
\end{eqnarray*}
Also clearly $\ell=\log 2$ and $\ell_\pi=0$. Hence $f(\ell)=\sum_{n=1}^\infty \frac{2^{n+1}}{2^{2n}}=2>1$ and $f_\pi(\ell_\pi)=\sum_{n=1}^\infty 1>1.$ In particular, Theorem \ref{fgeq1} is applicable and we deduce that $\bd \Pi(\Sigma)=\frac{h(\pi\Sigma)}{\log \m} + \frac{h(\Sigma) - h(\pi\Sigma)}{\log \n}$.
\end{comment}
%%%%%%%%%%%%%%%%%%%%%%%%%%%%%%

\subsubsection{Example whose box dimension equals the trivial lower bound}\label{sdimdrop}

Fix $\m \geq 2$ and $\n \geq \max\{\m+1,5\}$. Let 
$$\I=\{(1,1), (1,2), (1,3), (1,4), (1,5), (2,1)\}$$
and $\Omega=\{(1,3), (1,4), (1,5)\}$. Put 
$$\C=\{(1,1)\}\cup\{(2,1)\} \cup\{w\, (1,2)^m \; : \; w \in \Omega^*, \; m \geq 2^{|w|}\}$$
where $(1,2)^m$ denotes the concatenation of $m$ instances of the digit $(1,2)$, and let $\Sigma$ be the coded subshift generated by $\C$. Note that $\pi \Sigma=\{1,2\}^\N$. It is easy to see that $h_\pi=h(\pi \Sigma)=\log 2$, and we will show that $h \leq \log 2 < \log 3=h(\Sigma)$, see Lemma \ref{I_estimate2}. The graph $G$ (see Figure \ref{coded1}) presents $\Sigma$. We will be interested in words which label a path that begins and ends at the vertex $v$. 

\begin{figure}[H]
\begin{center}
   \begin{tikzpicture}[
roundnode/.style={circle, draw=black!60, fill=black!5, thick, minimum size=2mm},]
        \node [roundnode] (1) {$v$} ;

        \path[->] (1) edge  [loop above] node[auto] {\tiny $(1,1)$} ();
        \path[->] (1) edge  [loop below] node[auto] {\tiny $(2,1)$} ();
        \path[->] (1) edge  [loop right] node[auto, xshift=-12pt, yshift=-8pt] {\small $\substack{ w\,(1,2)^m \\ (w \in \Omega^*, \; m \geq 2^{|w|})}$} ();

    \end{tikzpicture}
\end{center} \caption{The graph $G$} \label{coded1}
\end{figure}
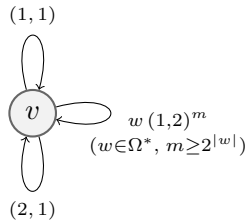

\begin{defn}
For each $n \in \N$ let $I_n$ denote all strings in $\Sigma_n$ which can be presented by a path on $G$ ending at $v$. Let $I=\bigcup_{n=1}^\infty I_n$.
\end{defn}

\begin{lma} We have
$$\limsup_{n \to \infty} \frac{1}{n} \log \# I_n \leq \log 2.$$
\label{I_estimate2}
\end{lma}

\begin{proof}
Suppose a word in $ I_n$ has $c$ digits from $\Omega$ and $a$ digits from $\{(1,1),(2,1)\}$. 

By definition of the code words $\C$, we must have $a+c+2^c \leq n$ therefore $c \leq \log_2(n-a).$ Now, assuming $c>0$, for each $1 \leq k \leq c$ there are ${c-1 \choose k-1}$ ways to divide the $c$ digits into $k$ groups.

 Following each of the $k$ blocks of digits from $\Omega$ there must be a string of $(1,2)$'s whose length is equal to the exponential of the length of that block. That leaves $n-c-2^c-a$ extra $(1,2)$'s to be distributed. These can be placed after any of the $k$ blocks of $(1,2)$'s, or directly before the first block of digits from $\Omega$. This gives ${n-c-2^c-a+k \choose k}$ different ways in which we can distribute the excess $(1,2)$'s. 

Finally, we can distribute the $a$ digits from $\{(1,1),(2,1)\}$ directly preceding any of the $k$ blocks of $(1,2)$'s or at the end of the word. This gives ${a+k \choose k}$ possibilities for distributing the $a$ digits from $\{(1,1),(2,1)\}$.

Note that since $k \leq c \leq \log_2 (n-a)$ we have
$${a +k \choose k} \leq {a+ \log_2 (n-a) \choose \log_2 (n-a)} \leq (e+en)^{\log_2 n}$$
where we have used that ${ N \choose k} \leq (\frac{eN}{k})^k$. Similarly
$${n-c-2^c-a+k \choose k} \leq {n-c-2^c-a+\log_2 (n-a) \choose \log_2 (n-a)} \leq  (e+en)^{\log_2 n}.$$
Also, since $c \leq \log_2 (n-a)$, 
$${c-1 \choose k-1} \leq 2^{ \log_2 (n-a)} \leq 2^{\log_2 n}$$
where we have first bounded ${c-1 \choose k-1}$ by the central binomial term and used that ${2N \choose N} \leq 4^N$.

Therefore
\begin{eqnarray*}
\# I_n & \leq& \sum_{a=0}^n \sum_{c=0}^{\log_2 (n-a)} \sum_{k=1}^c {c-1 \choose k-1} 3^c {n-c-2^c-a+k \choose k} {a+k \choose k} 2^a \\
 & \leq&(e+en)^{2\log_2 n} 2^{\log n}\sum_{a=0}^n \sum_{c=0}^{\log_2(n-a)} \sum_{k=1}^c 3^c 2^a \\
&\leq& (e+en)^{2\log_2 n} 2^{\log_2 n} 3^{\log_2 n}\sum_{a=0}^n \sum_{c=0}^{\log_2(n-a)} \sum_{k=1}^c  2^a 
\end{eqnarray*}
from which the result follows.
\end{proof}

Using the above estimate for $\# I_n$, it is now easy to compute the entropy of $\Sigma$.

\begin{lma} \label{entropy2}
$h(\Sigma)= \log 3$.
\end{lma}

\begin{proof}
The lower bound $h(\Sigma)\geq \log 3$ follows from the fact that $\Omega^\N \subset \Sigma$. So it is sufficient to prove the upper bound. Fix any $\varepsilon>0$. Suppose $\i \in \Sigma_n$. Then $\i$ falls into one of the following mutually exclusive categories:
\begin{enumerate}[(i)]
\item  $\i \in I_n$.
\item $\i=\j\k$ for $\j \in I$ and $\k=w\, (1,2)^m$ where $w \in \Omega^*$, $m \geq 0$.
\item $\i=w\, (1,2)^m$ for $w \in \Omega^*$ and $m \geq 0$. 
\end{enumerate}
By Lemma \ref{I_estimate2} the number of strings in category (i) is $\lesssim_\varepsilon 2^{(1+\varepsilon)n}$. The number of strings in category (iii) is given by
$$\sum_{m=0}^n 3^{n-m} \lesssim_\varepsilon 3^{(1+\varepsilon)n}.$$
 Finally, the number of strings in category (ii) is given by
$$\sum_{j=1}^{n-1} \sum_{m=0}^{n-j-1} \# I_j 3^{n-m-j} \leq \sum_{j=1}^{n-1} \sum_{m=0}^{n-j-1}  2^j3^{n-m-j}\lesssim_\varepsilon 3^{(1+\varepsilon)n}.$$
Hence $\frac{1}{n}\log \# \Sigma_n \lesssim_\varepsilon (1+\varepsilon) \log 3$ which concludes the proof of the upper bound since $\varepsilon>0$ was chosen arbitrarily.
\end{proof}

% For $\delta>0$ let $l \approx \frac{\log \delta}{-\log m}$, $k= \frac{\log \delta}{-\log n}$. 

We will now prove Theorem \ref{dimdrop} by showing that
$$\bd \Pi(\Sigma)= \max\left\{\frac{\log 3}{\log \n}, \frac{\log 2}{\log \m}\right\}=\max\left\{\frac{h(\Sigma)}{\log \n}, \frac{h(\pi \Sigma)}{\log \m}\right\}.$$

Note that $\bd \Pi(\Sigma)$ can attain either $\frac{h(\Sigma)}{\log \n}$ or $ \frac{h(\pi \Sigma)}{\log \m}$. For instance if $\n=5$, $\m=2$ then $\bd \Pi(\Sigma)=1=\frac{h(\pi \Sigma)}{\log \m}$. Whereas if $\n=6$, $\m=5$ then $\bd \Pi(\Sigma)= \frac{\log 3}{\log 6}= \frac{h(\Sigma)}{\log \n}$.

\begin{proof}[Proof of Theorem \ref{dimdrop}] The lower bound corresponds to the trivial lower bound from Proposition \ref{bounds}. So we just need to prove the upper bound. Fix $\varepsilon>0$, $\delta>0$. Let $k=k(\delta)$ and $l=l(\delta)$ and $\i \in \Sigma_l$. Then $\i$ falls into one of the following mutually exclusive categories.
\begin{enumerate}[(1)]
\item $\i=\j\k$ where $\j \in \Sigma_k$ and $\pi(\k)=1^{l-k}$. 
\item  $\i=\j\k \, (2,1) \,\l$ where:  for some $1 \leq m \leq l-k$, $\pi(\l) \in \{1,2\}^{m-1}$; $\pi(\k)=1^{l-k-m}$; $\j \in \Sigma_k$ has the form $\j=uw$ for $u \in I$ and $w \in \Omega^*$ with length $1 \leq |w| \leq \log_2 (l-k-m).$
\item $\i=\j \k \, (2,1) \, \l$ where: for some $1 \leq m \leq l-k$, $\pi(\l) \in \{1,2\}^{m-1}$; $\pi(\k)=1^{l-k-m}$; $\j= uw\, (1,2)^z$ where $u \in I$, $1 \leq z \leq k$, $w \in \Omega^*$ with length $0\leq |w| \leq \log_2 (l-k+z-m)$.
\end{enumerate}
For each $j=1,2,3$ we define 
$$A_j:=\bigcup_{\textnormal{$\i \in \Sigma_l$ in category ($j$)}} \Pi([\i]).$$ 
Then 
\begin{eqnarray}
N_\delta(\Pi(\Sigma)) \leq \sum_{j=1}^3 N_\delta(A_j).\label{a_i}\end{eqnarray}

Firstly,
$$N_\delta(A_1)=\# \Sigma_k \lesssim_\varepsilon 3^{(1+\varepsilon)k}$$
by Lemma \ref{entropy2}. Secondly,
\begin{eqnarray*}
N_\delta(A_2)=\sum_{m=1}^{l-k} \sum_{|w|=1}^{\log_2 (l-k-m)} \# I_{k-|w|} 3^{|w|} 2^{m-1} 
&\lesssim_\varepsilon& \sum_{m=1}^{l-k} \sum_{|w|=1}^{\log_2 (l-k-m)} 2^{(1+\varepsilon)(k-|w|)}3^{|w|} 2^{m-1}\\
&\lesssim_\varepsilon& 2^{(1+\varepsilon)k} \left(\frac{3}{2^{1+\varepsilon}}\right)^{\log_2 (l-k)} 2^{(1+\varepsilon)(l-k)} \lesssim_\varepsilon 2^{(1+2\varepsilon)l}.
\end{eqnarray*}
Finally,  
\begin{eqnarray*}
N_\delta(A_3) &=& \sum_{m=1}^{l-k} \sum_{z=1}^k \sum_{|w|=0}^{\log_2 (l-k+z-m)} \# I_{k-|w|-z} 3^{|w|} 2^{m-1} \\
&\lesssim_\varepsilon& \sum_{m=1}^{l-k} \sum_{z=1}^k \sum_{|w|=0}^{\log_2 (l-k+z-m)} 2^{(1+\varepsilon)( k-|w|-z)} 3^{|w|} 2^{m-1} \\
&\lesssim_\varepsilon&  \sum_{m=1}^{l-k} \sum_{z=1}^k 2^{(1+\varepsilon)(k-z)}\left(\frac{3}{2^{1+\varepsilon}}\right)^{\log_2 (l-k+z-m)}2^{m-1} \\
&\lesssim_\varepsilon&    2^{(1+\varepsilon)k} \left(\frac{3}{2^{1+\varepsilon}}\right)^{\log_2 l}2^{(1+\varepsilon)(l-k)} \lesssim_\varepsilon 2^{(1+2\varepsilon)l} .
\end{eqnarray*}
By \eqref{a_i} we deduce that 
$$\bd \Pi(\Sigma) \leq\max\left\{\frac{\log 3}{\log \n}, \frac{\log 2}{\log \m}\right\},$$
as required.
\end{proof}


\begin{bibdiv}
\begin{biblist}

\bib{bedford}{thesis}{,
  author={Bedford, Tim},
  title={Crinkly curves, Markov partitions and dimension},
  date={1984},
  school={University of Warwick}
}

\bib{blanchard}{article}{
   author={Blanchard, F.},
 title={$\beta$-expansions and symbolic dynamics},
   journal={Theoret. Comput. Sci.},
   volume={65},
   date={1989},
   number={2},
   pages={131--141}
}

\bib{bh}{article}{
   author={Blanchard, F.},
   author={Hansel, G.},
   title={Syst\`emes cod\'{e}s},
   journal={Theoret. Comput. Sci.},
   volume={44},
   date={1986},
   number={1},
   pages={17--49}
}

\bib{ct}{article}{
   author={Climenhaga, Vaughn},
   author={Thompson, Daniel J.},
   title={Intrinsic ergodicity beyond specification: $\beta$-shifts, $S$-gap
   shifts, and their factors},
   journal={Israel J. Math.},
   volume={192},
   date={2012},
   number={2},
   pages={785--817}
}

\bib{deliu}{article}{
   author={Deliu, Anca},
   author={Geronimo, J. S.},
   author={Shonkwiler, R.},
   author={Hardin, D.},
   title={Dimensions associated with recurrent self-similar sets},
   journal={Math. Proc. Cambridge Philos. Soc.},
   volume={110},
   date={1991},
   number={2},
   pages={327--336}
}

\bib{falconer}{book}{
    AUTHOR = {Falconer, Kenneth},
     TITLE = {Fractal geometry},
   EDITION = {Third},
      NOTE = {Mathematical foundations and applications},
 PUBLISHER = {John Wiley \& Sons, Ltd., Chichester},
      YEAR = {2014},
     PAGES = {xxx+368},
      ISBN = {978-1-119-94239-9},
   MRCLASS = {28-01 (11K55 28A78 28A80 37C45 37F10)},
  MRNUMBER = {3236784},
MRREVIEWER = {Manuel Mor\'{a}n},
}


\bib{furstenberg}{article}{
   author={Furstenberg, H.},
   title={Disjointness in ergodic theory, minimal sets, and a problem in Diophantine
approximation},
   journal={Math. Systems Theory},
   volume={1},
   date={1967},
   pages={1--49}
}

 

\bib{kp}{article}{
   author={Kenyon, R.},
   author={Peres, Y.},
   title={Hausdorff dimensions of sofic affine-invariant sets},
   journal={Israel J. Math.},
   volume={94},
   date={1996},
   pages={157--178}
}

\bib{kp-measures}{article}{
   author={Kenyon, R.},
   author={Peres, Y.},
   title={Measures of full dimension on affine-invariant sets},
   journal={Ergodic Theory Dynam. Systems},
   volume={16},
   date={1996},
   number={2},
   pages={307--323}
}

\bib{lm}{book}{
   author={Lind, Douglas},
   author={Marcus, Brian},
   title={An introduction to symbolic dynamics and coding},
   publisher={Cambridge University Press, Cambridge},
   date={1995},
   pages={xvi+495}
}

\bib{mcmullen}{article}{
   author={McMullen, Curt},
   title={The Hausdorff dimension of general Sierpi\'{n}ski carpets},
   journal={Nagoya Math. J.},
   volume={96},
   date={1984},
   pages={1--9}
}


\bib{pavlov}{article}{
   author={Pavlov, R.},
   title={On entropy and intrinsic ergodicity of coded subshifts.},
   journal={to appear in Proc. Amer. Math. Soc.}
   eprint={https://arxiv.org/abs/1803.05966}
}

\bib{vj}{article}{
   author={Vere-Jones, D.},
   title={Geometric ergodicity in denumerable Markov chains},
   journal={Quart. J. Math. Oxford Ser. (2)},
   volume={13},
   date={1962},
   pages={7--28},
}



\end{biblist}
\end{bibdiv}
\end{document}